\documentclass[11pt]{amsart}
\usepackage[margin=1in]{geometry}
\usepackage{latexsym}
\usepackage{amsfonts}
\usepackage{amsmath}
\usepackage{amssymb}
\usepackage{amsthm}
\usepackage{enumerate}
\usepackage[hang,flushmargin]{footmisc}
\usepackage{caption}
\usepackage{tabu}
\usepackage{mathrsfs}
\usepackage{subfig}
\usepackage[foot]{amsaddr}
\usepackage{float}

\usepackage{graphicx}
\usepackage{epstopdf}
\usepackage{epsfig}
\usepackage{caption}
 
\usepackage{bm} 
\usepackage{cite}

\DeclareMathOperator{\VHC}{\mathsf{VHC}}

\newtheorem{theorem}{Theorem}[section]
\newtheorem{lemma}{Lemma}[section]
\newtheorem{proposition}{Proposition}[section]
\newtheorem{corollary}{Corollary}[section]

\newtheorem{conjecture}{Conjecture}[section]
\newtheorem{question}{Question}[section]

\theoremstyle{definition}
\newtheorem{definition}{Definition}[section]

\newtheorem{example}{Example}[section]

\begin{document}
\title{Fertility Numbers}

\author{Colin Defant$^1$}
\address{$^1$Princeton University}
\email{cdefant@princeton.edu}

\begin{abstract}
A nonnegative integer is called a \emph{fertility number} if it is equal to the number of preimages of a permutation under West's stack-sorting map. We prove structural results concerning permutations, allowing us to deduce information about the set of fertility numbers. In particular, the set of fertility numbers is closed under multiplication and contains every nonnegative integer that is not congruent to $3$ modulo $4$. We show that the lower asymptotic density of the set of fertility numbers is at least $1954/2565\approx 0.7618$. We also exhibit some positive integers that are not fertility numbers and conjecture that there are infinitely many such numbers.   
\end{abstract}

\maketitle

\bigskip

\section{Introduction}

Throughout this article, the word ``permutation" refers to a permutation of a finite set of positive integers. We write permutations as words in one-line notation. Let $S_n$ denote the set of permutations of $\{1,\ldots,n\}$. We say a permutation is \emph{normalized} if it an element of $S_n$ for some $n$ (e.g., the permutation $12547$ is not normalized). 

The study of permutation patterns, which has now developed into a vast area of research, began with Knuth's investigation of stack-sorting in \cite{Knuth}. In his 1990 Ph.D. thesis, Julian West \cite{West} explored a deterministic variant of Knuth's stack-sorting algorithm, which we call the \emph{stack-sorting map}. This map, denoted $s$, is defined as follows.

Assume we are given an input permutation $\pi=\pi_1\cdots\pi_n$. Throughout this algorithm, if the next entry in the input permutation is smaller than the entry at the top of the stack or if the stack is empty, the next entry in the input permutation is placed at the top of the stack. Otherwise, the entry at the top of the stack is annexed to the end of the growing output permutation. This procedure stops when the output permutation has length $n$. We then define $s(\pi)$ to be this output permutation. Figure \ref{Fig1} illustrates this procedure and shows that $s(4162)=1426$.  

\begin{figure}[h]
\begin{center}
\includegraphics[width=1\linewidth]{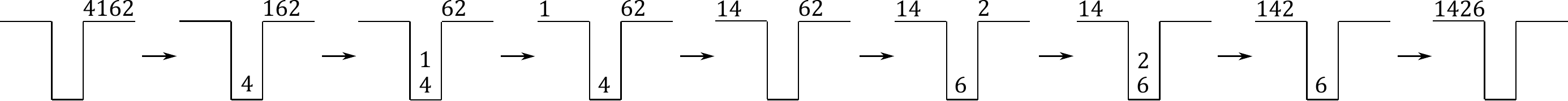}
\end{center}  
\caption{The stack-sorting map $s$ sends $4162$ to $1426$.}
\end{figure}\label{Fig1}

There is an alternative recursive description of the stack-sorting map. Specifically, if $m$ is the largest entry appearing in the permutation $\pi$, we can write $\pi=LmR$, where $L$ and $R$ are the substrings of $\pi$ appearing to the left and right of $m$, respectively. Then $s(\pi)=s(L)s(R)m$. For example, $s(4162)=s(41)s(2)6=s(41)26=s(1)426=1426$. It is also possible to describe the stack-sorting algorithm in terms of in-order readings and postorder readings of decreasing binary plane trees \cite{Bona, Defant}. 

West defined the \emph{fertility} of a permutation $\pi$ to be $|s^{-1}(\pi)|$, the number of preimages of $\pi$ under the stack-sorting map \cite{West}. He proceeded to compute the fertilities of the permutations of the forms \[23\cdots k1(k+1)\cdots n,\quad 12\cdots(k-2)k(k-1)(k+1)\cdots n,\quad\text{and}\quad k12\cdots(k-1)(k+1)\cdots n.\] Bousquet-M\'elou then defined a \emph{sorted} permutation to be a permutation that has positive fertility \cite{Bousquet}; she provided an algorithm for determining whether or not a given permutation is sorted. She also mentioned that it would be interesting to find a method for computing the fertility of any given permutation. The current author found such a method in \cite{Defant}. In fact, the results in that paper are even more general; they allow one to enumerate certain types of decreasing plane trees that have a given permutation as their postorder readings. The current author has since used this method to improve the best-known upper bounds for the enumeration of so-called $3$-stack-sortable and $4$-stack-sortable permutations in \cite{Defant2}. See \cite{Bona, BonaSurvey, Defant2, Zeilberger} for more information about $t$-stack-sortable permutations. 

The method developed in \cite{Defant} and \cite{Defant2} for computing fertilities makes use of new combinatorial objects called \emph{valid hook configurations}. The authors of \cite{Defant3} gave a concise description of valid hook configurations and exhibited a bijection between these objects and certain ordered pairs of set partitions and acyclic orientations. They then exploited this bijection to study permutations with fertility $1$, showing that these permutations\footnote{These permutations are called \emph{uniquely sorted}. They are studied further in \cite{DefantCatalan} and \cite{Hanna}} are counted by an interesting sequence known as Lassalle's sequence (which Lassalle introduced in \cite{Lassalle}). This bijection also allowed the authors to connect cumulants arising in free probability theory with valid hook configurations and the stack-sorting map (building upon results from \cite{Josuat}). For completeness, we repeat the short description of valid hook configurations from \cite{Defant3} in Section 2. See also \cite{DefantCatalan, DefantMotzkin, Defant3, Hanna, Maya} for further investigation of the combinatorics of valid hook configurations.

\begin{definition}\label{Def1}
Say a nonnegative integer $f$ is a \emph{fertility number} if there exists a permutation with fertility $f$. Say a nonnegative integer is an \emph{infertility number} if it is not a fertility number. 
\end{definition}

For example, $0,1$, and $2$ are fertility numbers because $|s^{-1}(21)|=0$, $|s^{-1}(1)|=1$, and $|s^{-1}(12)|=2$. In Section 3, we prove the following statements about fertility numbers. These are Theorems \ref{Thm1}--\ref{Thm7} below.    

\begin{itemize}

\item The set of fertility numbers is closed under multiplication. 

\item If $f$ is a fertility number, then there are arbitrarily long permutations with fertility $f$. 

\item Every nonnegative integer that is not congruent to $3$ modulo $4$ is a fertility number. The lower asymptotic density of the set of fertility numbers is at least $1954/2565\approx 0.7618$. 

\item The smallest fertility number that is congruent to $3$ modulo $4$ is $27$.  

\item If $f$ is a positive fertility number, then there exist a positive integer $n\leq f+1$ and a permutation $\pi\in S_n$ such that $f=|s^{-1}(\pi)|$. 
\end{itemize}

The fourth bullet point above shows, in particular, that the notion of a fertility number is not pointless because infertility numbers exist. The fifth bullet shows that determining whether or not a given number is a fertility number can be reduced to a finite search. This finite search can be very long, but we will see in our proof of the fourth bullet point that we can often cut corners to reduce the computations. In Section 4, we give suggestions for future work, including three conjectures.  

\section{Valid Hook Configurations}

In this section, we review some of the theory of valid hook configurations. Our presentation is virtually the same as that given in \cite{Defant3}, but we include it here for completeness. It is important to note that the valid hook configurations defined below are, strictly speaking, different from those defined in \cite{Defant} and \cite{Defant2}. For a lengthier discussion of this distinction, see \cite{Defant3}. 

The construction of a valid hook configuration commences with the choice of a permutation $\pi=\pi_1\cdots\pi_n$. A \emph{descent} of $\pi$ is an index $i$ such that $\pi_i>\pi_{i+1}$. Let $d_1<\cdots<d_k$ be the descents of $\pi$. We use the example permutation $3142567$ to illustrate the construction. The \emph{plot} of $\pi$ is the graph displaying the points $(i,\pi_i)$ for $1\leq i\leq n$. The left image in Figure \ref{Fig2} shows the plot of our example permutation. A point $(i,\pi_i)$ is a \emph{descent top} if $i$ is a descent. The descent tops in our example are $(1,3)$ and $(3,4)$. 

\begin{figure}[t]
  \centering
  \subfloat[]{\includegraphics[width=0.2\textwidth]{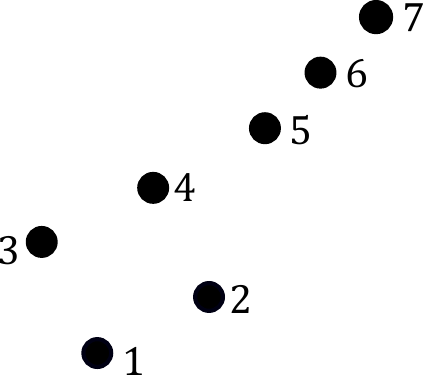}}
  \hspace{1.5cm}
  \subfloat[]{\includegraphics[width=0.2\textwidth]{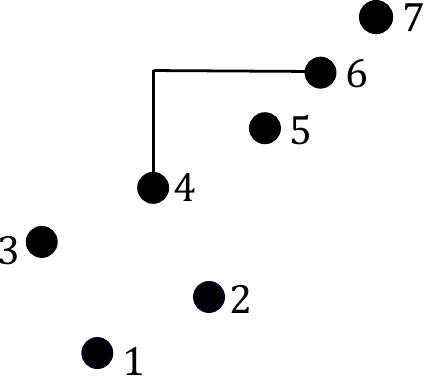}}
  \caption{The left image is the plot of $3142567$. The right images shows this plot along with a single hook.}\label{Fig2}
\end{figure}

A \emph{hook} of $\pi$ is drawn by starting at a point $(i,\pi_i)$ in the plot of $\pi$, moving vertically upward, and then moving to the right until reaching another point $(j,\pi_j)$. We must necessarily have $i<j$ and $\pi_i<\pi_j$. The point $(i,\pi_i)$ is called the \emph{southwest endpoint} of the hook, while $(j,\pi_j)$ is called the \emph{northeast endpoint}. The right image in Figure \ref{Fig2} shows our example permutation with a hook that has southwest endpoint $(3,4)$ and northeast endpoint $(6,6)$. 

A \emph{valid hook configuration} of $\pi$ is a configuration of hooks drawn on the plot of $\pi$ subject to the following constraints: 

\begin{enumerate}[1.]
\item The southwest endpoints of the hooks are precisely the descent tops of the permutation. 

\item A point in the plot cannot lie directly above a hook. 

\item Hooks cannot intersect each other except in the case that the northeast endpoint of one hook is the southwest endpoint of the other. 
\end{enumerate}  

\begin{figure}[t]
\begin{center}
\includegraphics[width=.7\linewidth]{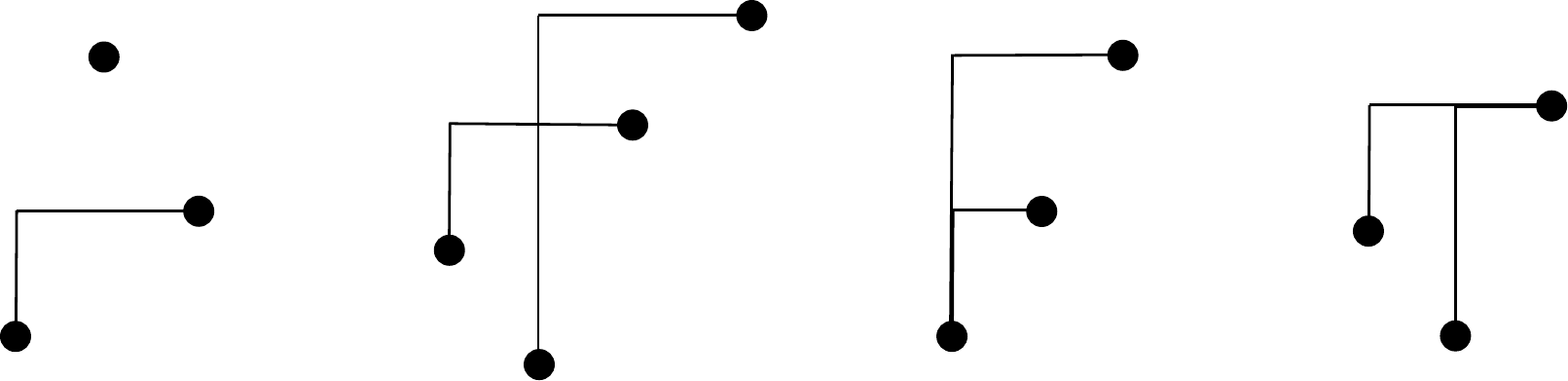}
\caption{Four configurations of hooks that are forbidden in a valid hook configuration.}
\label{Fig3}
\end{center}  
\end{figure}

\begin{figure}[t]
\begin{center}
\includegraphics[width=.7\linewidth]{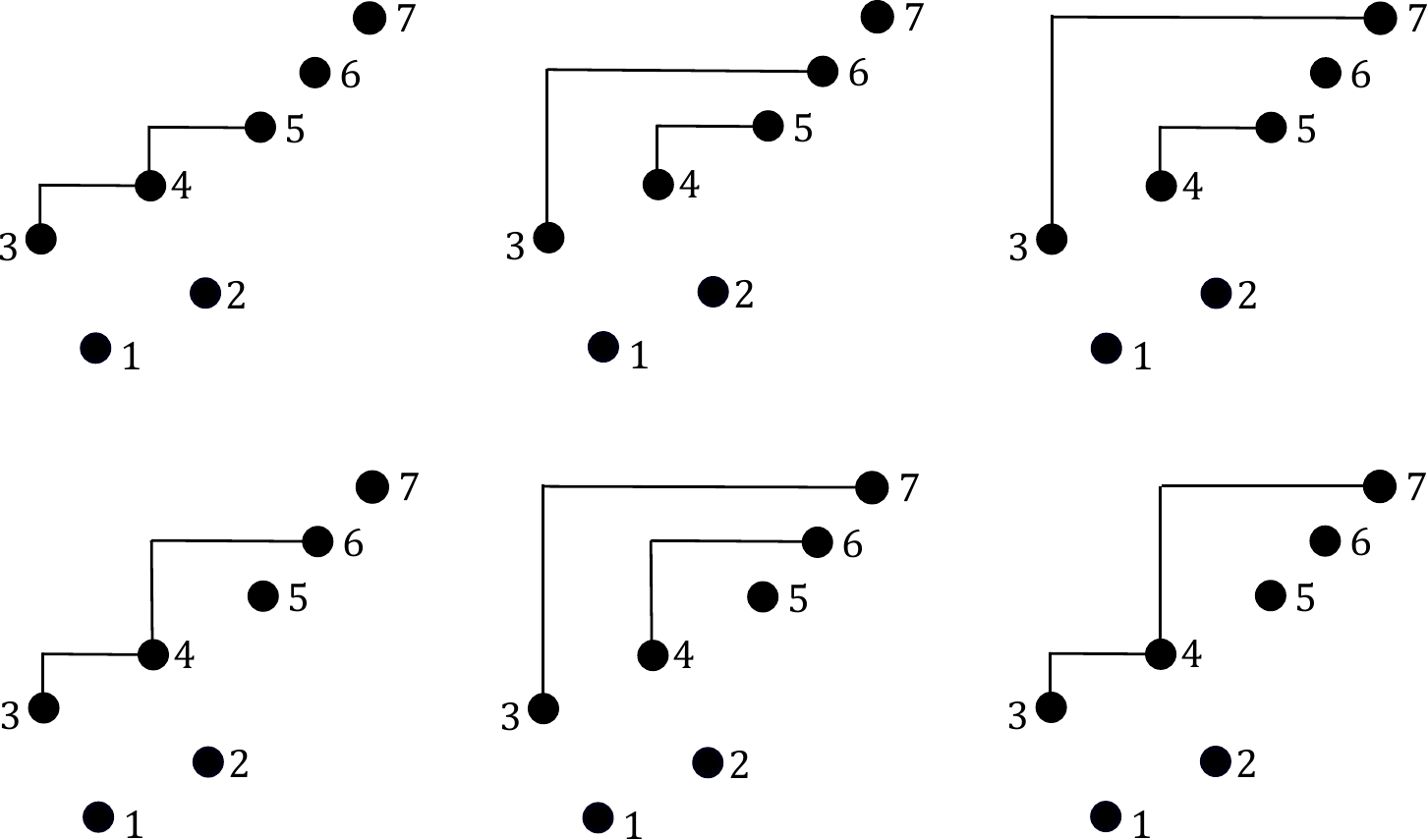}
\caption{All of the valid hook configurations of $3142567$.}
\label{Fig4}
\end{center}  
\end{figure}

Figure \ref{Fig3} shows four placements of hooks that are forbidden by conditions 2 and 3.
Figure \ref{Fig4} shows all of the valid hook configurations of $3142567$. 
Note that the total number of hooks in a valid hook configuration of $\pi$ is exactly $k$, the number of descents of $\pi$. Because the southwest endpoints of the hooks are the points $(d_i,\pi_{d_i})$, we have a natural ordering of the hooks. Namely, the $i^\text{th}$ hook is the hook whose southwest endpoint is $(d_i,\pi_{d_i})$. We can write a valid hook configuration of $\pi$ concisely as a $k$-tuple $\mathcal H=(H_1,\ldots,H_k)$, where $H_i$ is the $i^\text{th}$ hook. 

A valid hook configuration of $\pi$ induces a coloring of the plot of $\pi$. To begin the process of coloring the plot, draw a ``sky" over the entire diagram. As one might expect, we color the sky blue. Assign arbitrary distinct colors other than blue to the $k$ hooks in the valid hook configuration. 

There are $k$ northeast endpoints of hooks, and these points remain uncolored. However, all of the other $n-k$ points will be colored. In order to decide how to color a point $(i,\pi_i)$ that is not a northeast endpoint, imagine that this point looks directly upward. If this point sees a hook when looking upward, it receives the same color as the hook that it sees. If the point does not see a hook, it must see the sky, so it receives the color blue. However, if $(i,\pi_i)$ is the southwest endpoint of a hook, then it must look around (on the left side of) the vertical part of that hook. See Figure \ref{Fig5} for the colorings induced by the valid hook configurations in Figure \ref{Fig4}. Note that the leftmost point $(1,3)$ is blue in each of these colorings because this point looks around the first (red) hook and sees the sky. 

To summarize, we started with a permutation $\pi$ with exactly $k$ descents. We chose a valid hook configuration of $\pi$ by drawing $k$ hooks according to the rules 1, 2, and 3 above. This valid hook configuration then induced a coloring of the plot of $\pi$. Specifically, $n-k$ points were colored, and $k+1$ colors were used (one for each hook and one for the sky). Let $q_i$ be the number of points colored the same color as the $i^\text{th}$ hook, and let $q_0$ be the number of points colored blue (sky color). Then $(q_0,q_1,\ldots,q_k)$ is a composition of $n-k$ into $k+1$ parts.\footnote{Throughout this article, a \emph{composition of $b$ into $a$ parts} is an $a$-tuple of positive integers that sum to $b$. For $i\in\{1,\ldots,k\}$, the number $q_i$ is positive because the point immediately to the right of the southwest endpoint of the $i^\text{th}$ hook is given the same color as the $i^\text{th}$ hook. The number $q_0$ is positive because $(1,\pi_1)$ is colored blue.} We call a composition obtained in this way a \emph{valid composition of }$\pi$. Let $\VHC(\pi)$ be the set of valid hook configurations of $\pi$. Let $\mathcal V(\pi)$ be the set of valid compositions of $\pi$. 

\begin{figure}[t]
\begin{center}
\includegraphics[width=.7\linewidth]{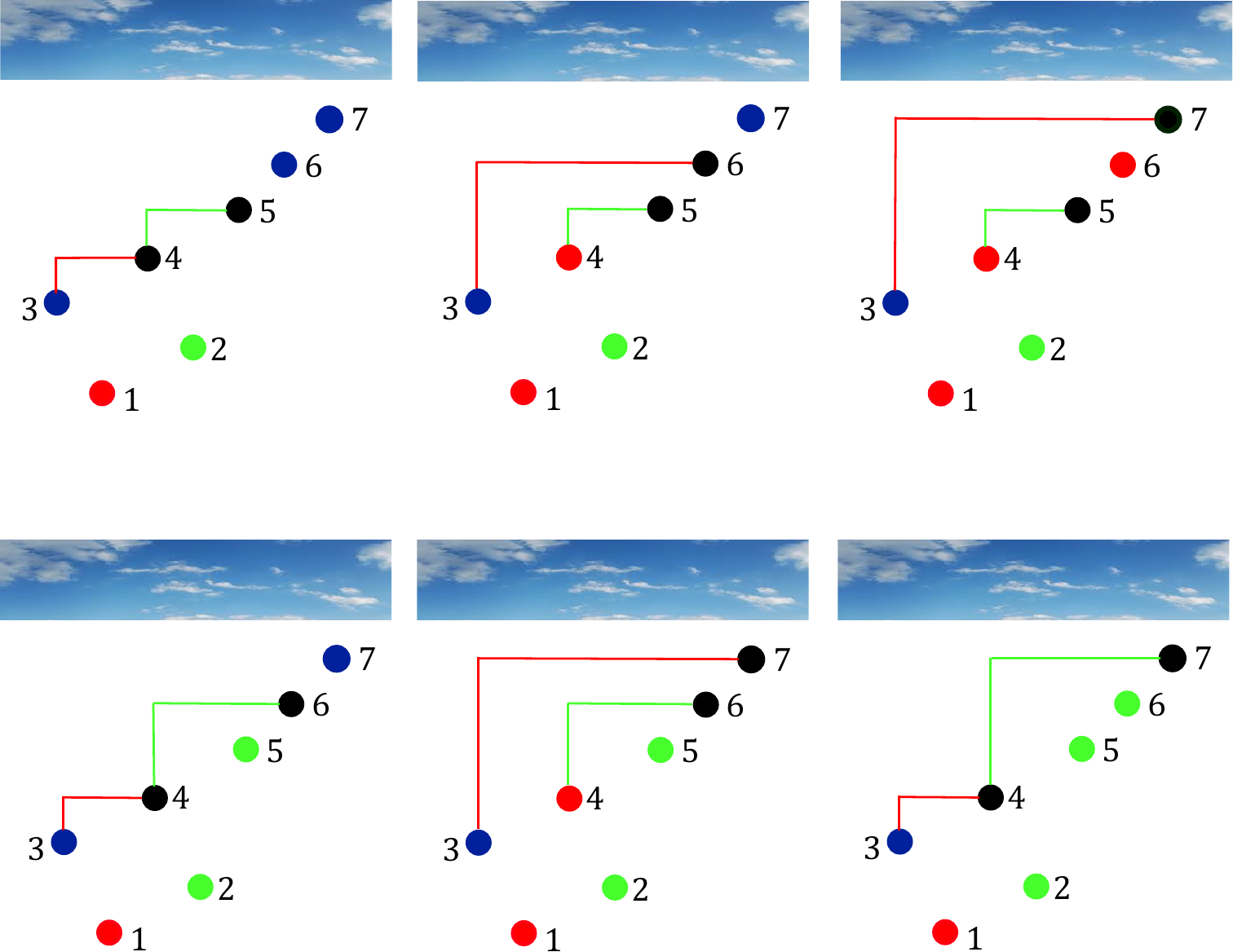}
\caption{The different colorings induced by the valid hook configurations of $3142567$.}
\label{Fig5}
\end{center}  
\end{figure}

The following theorem is the main reason why valid hook configurations are so useful when studying the stack-sorting map. Let $C_j=\frac{1}{j+1}{2j\choose j}$ denote the $j^\text{th}$ Catalan number. We will find it convenient to introduce the notation \[C_{(q_0,\ldots,q_k)}=\prod_{t=0}^kC_{q_t}\] for any composition $(q_0,\ldots,q_k)$. 

\begin{theorem}[\!\!\cite{Defant}]\label{Thm5}
If $\pi$ has exactly $k$ descents, then the fertility of $\pi$ is given by the formula \[|s^{-1}(\pi)|=\sum_{(q_0,\ldots,q_k)\in\mathcal V(\pi)}C_{(q_0,\ldots,q_k)}.\] 
\end{theorem}

Note in particular that a permutation is sorted if and only if it has a valid hook configuration. See \cite{Defant,Defant2,Defant3} for extensions and refinements of Theorem \ref{Thm5}. 

\begin{example}\label{Exam1}
The permutation $\pi=3142567$ has six valid hook configurations, which are shown in Figure \ref{Fig4}. The colorings induced by these valid hook configurations are portrayed in Figure \ref{Fig5}. The valid compositions of these valid hook configurations are (reading the first row before the second row, each from left to right) \[(3,1,1),\quad (2,2,1),\quad(1,3,1),\quad(2,1,2),\quad(1,2,2),\quad(1,1,3).\]It follows from Theorem \ref{Thm5} that \[|s^{-1}(\pi)|=C_{(3,1,1)}+C_{(2,2,1)}+C_{(1,3,1)}+C_{(2,1,2)}+C_{(1,2,2)}+C_{(1,1,3)}=27.\] Consequently, $27$ is a fertility number.  
\end{example}

Throughout this paper, we implicitly make use of the following result, which is Lemma 3.1 in \cite{Defant2}. 

\begin{theorem}[\!\!\cite{Defant2}]\label{Thm6}
Let $\pi$ be a permutation. The map $\VHC(\pi)\to\mathcal V(\pi)$ sending each valid hook configuration of $\pi$ to its induced valid composition is injective. 
\end{theorem}

\section{Proofs of the Main Theorems}
We now exploit the valid hook configurations discussed in the previous section to prove our main theorems concerning fertility numbers. Let us begin with some useful definitions. 

Let $\pi=\pi_1\cdots\pi_n$ be a permutation. Let $H$ be a hook in a valid hook configuration of $\pi$ with southwest endpoint $(i,\pi_i)$ and northeast endpoint $(j,\pi_j)$. When referring to a point ``below" $H$, we mean a point $(x,y)$ with $i<x<j$ and $y<\pi_j$. In particular, the endpoints of a hook do not lie below that hook. 

\begin{definition}\label{Def2}
Let $\pi=\pi_1\cdots\pi_n$ be a permutation, and let $H$ be a hook drawn on the plot of $\pi$. We say $H$ is a \emph{stationary hook} if it appears in every valid hook configuration of $\pi$. 
\end{definition}

For example, suppose $\pi\in S_n$, $\pi_n=n$ and $\pi_i=n-1$, where $i\leq n-2$. Let $H$ be the hook with southwest endpoint $(i,n-1)$ and northeast endpoint $(n,n)$. The point $(i,n-1)$ is a descent top of $\pi$, so every valid hook configuration of $\pi$ must have a hook whose southwest endpoint is $(i,n-1)$. The northeast endpoint of such a hook must be $(n,n)$, so it follows that $H$ is a stationary hook of $\pi$. One can check that the hook drawn in Figure \ref{Fig6} is another example of a stationary hook. 

\begin{figure}[t]
\begin{center}
\includegraphics[width=.4\linewidth]{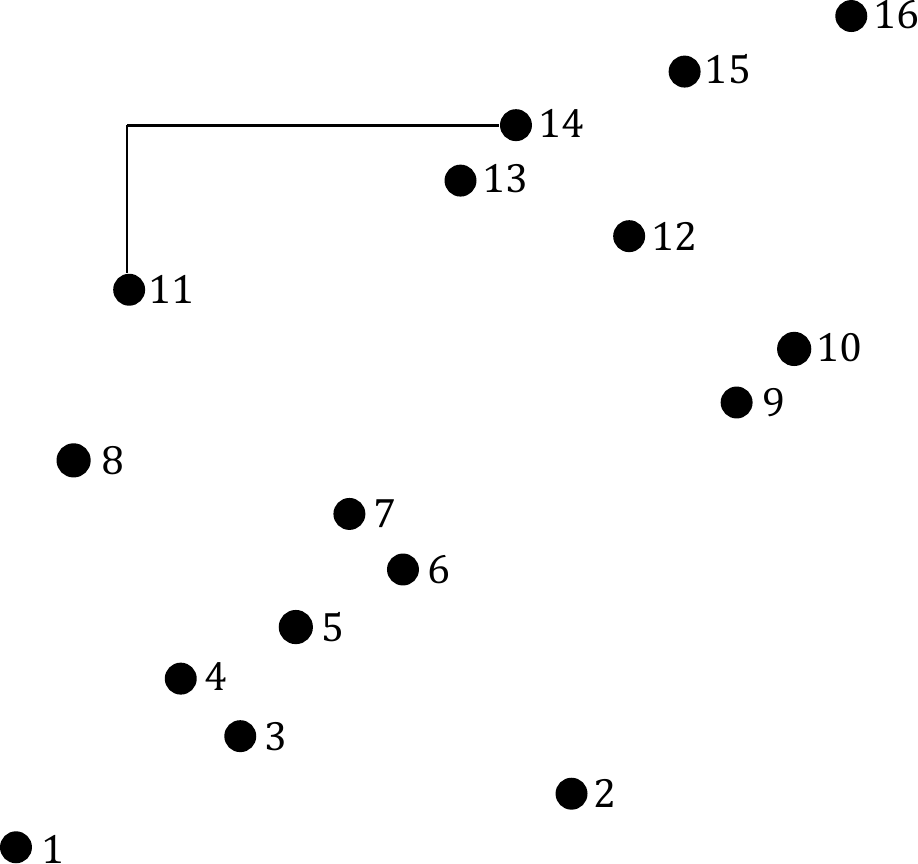}
\caption{A stationary hook of the permutation $1\,\,8\,\,11\,\,4\,\,3\,\,5\,\,7\,\,6\,\,13\,\,14\,\,2\,\,12\,\,15\,\,9\,\,10\,\,16$.}
\label{Fig6}
\end{center}  
\end{figure}

\begin{proposition}\label{Prop1}
Let $\pi=\pi_1\cdots\pi_n$ be a permutation with a stationary hook $H$. Let $(i,\pi_i)$ and $(j,\pi_j)$ be the southwest and northeast endpoints of $H$, respectively. Let $\sigma=\pi_1\cdots\pi_{i+1}\pi_j\cdots\pi_n$ and $\tau=\pi_{i+1}\cdots\pi_{j-1}$. We have
\[|s^{-1}(\pi)|=|s^{-1}(\sigma)||s^{-1}(\tau)|.\]
\end{proposition}

\begin{proof}
There is a natural bijection \[\VHC(\sigma)\times \VHC(\tau)\to\VHC(\pi)\] obtained by combining a valid hook configuration of $\sigma$ and a valid hook configuration of $\tau$ into a valid hook configuration of $\pi$. Furthermore, the colorings of the plots of $\sigma$ and $\tau$ combine into one coloring of $\pi$. Note that the non-blue colors used to color $\sigma$ must be different from those used to color $\tau$. The blue points in the plot of $\tau$ must change to the color of $H$ in the plot of $\pi$. See Figure \ref{Fig10} for a depiction of this combination of valid hook configurations and induced colorings. In that figure, $H$ is the hook with southwest endpoint $(3,11)$ and northeast endpoint $(10,14)$. 

Let $k_\sigma=\text{des}(\sigma)$ and $k_\tau=\text{des}(\tau)$ be the number of descents of $\sigma$ and the number of descents of $\tau$, respectively. Note that $H$ is a stationary hook of $\sigma$. If $i$ is the $r^\text{th}$ descent of $\sigma$, then every valid composition of $\sigma$ is of the form $(q_0,\ldots,q_{r-1},1,q_{r+1},\ldots,q_{k_\sigma})$. It follows from the above paragraph that the map $\mathcal V(\sigma)\times\mathcal V(\tau)\to\mathcal V(\pi)$ given by \[((q_0,\ldots,q_{r-1},1,q_{r+1},\ldots,q_{k_\sigma}),(q_0',\ldots,q_{k_\tau}'))\mapsto(q_0,\ldots,q_{r-1},q_0',\ldots,q_{k_\tau}',q_{r+1},\ldots,q_{k_\sigma})\] is a bijection. Invoking Theorem \ref{Thm5}, we find that 
\[|s^{-1}(\pi)|=\sum_{(q_0,\ldots,q_{r-1},1,q_{r+1},\ldots,q_{k_\sigma})\in\mathcal V(\sigma)}\:\sum_{(q_0',\ldots,q_{k_\tau}')\in\mathcal V(\tau)}C_{(q_0,\ldots,q_{r-1},q_0',\ldots,q_{k_\tau}',q_{r+1},\ldots,q_{k_\sigma})}\]

\[=\sum_{(q_0,\ldots,q_{r-1},1,q_{r+1},\ldots,q_{k_\sigma})\in\mathcal V(\sigma)}\:\sum_{(q_0',\ldots,q_{k_\tau}')\in\mathcal V(\tau)}C_{(q_0,\ldots,q_{r-1},1,q_{r+1},\ldots,q_{k_\sigma})}C_{(q_0',\ldots,q_{k_\tau}')}\] \[=\left[\sum_{(q_0,\ldots,q_{r-1},1,q_{r+1},\ldots,q_{k_\sigma})\in\mathcal V(\sigma)}C_{(q_0,\ldots,q_{r-1},1,q_{r+1},\ldots,q_{k_\sigma})}\right]\left[\sum_{(q_0',\ldots,q_{k_\tau}')\in\mathcal V(\tau)}C_{(q_0',\ldots,q_{k_\tau}')}\right]\] \[=|s^{-1}(\sigma)||s^{-1}(\tau)|.\qedhere\]
\end{proof}

\begin{figure}[t]
\begin{center}
\includegraphics[width=1\linewidth]{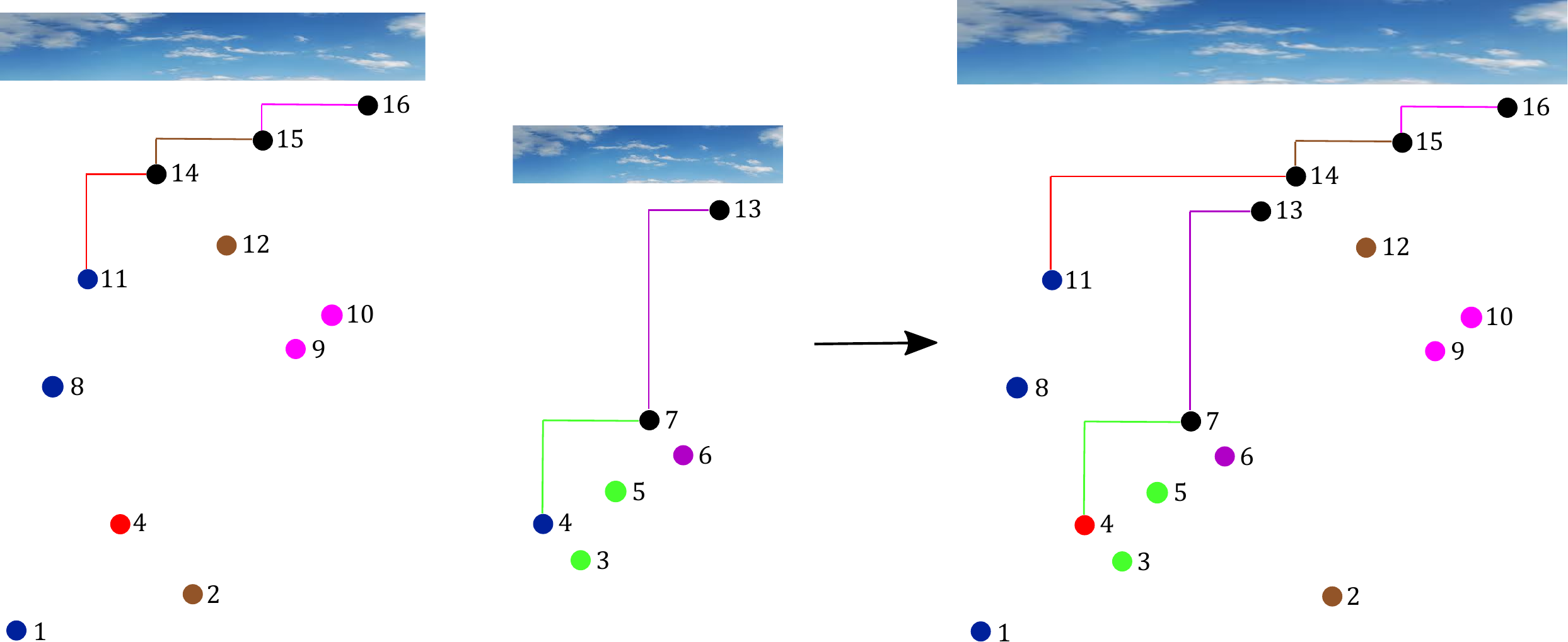}
\caption{Valid hook configurations of $\sigma=1\,\,8\,\,11\,\,4\,\,14\,\,2\,\,12\,\,15\,\,9\,\,10\,\,16$ and $\tau=4\,\,3\,\,5\,\,7\,\,6\,\,13$ combine to form a valid hook configuration of $\pi=1\,\,8\,\,11\,\,4\,\,3\,\,5\,\,7\,\,6\,\,13\,\,14\,\,2\,\,12\,\,15\,\,9\,\,10\,\,16$. In Proposition \ref{Prop1}, we consider a stationary hook $H$ of $\pi$. In this example, $H$ is the (red) hook with southwest endpoint $(3,11)$ and northeast endpoint $(10,14)$.}
\label{Fig10}
\end{center}  
\end{figure}

The following corollary allows us to explicitly construct permutations with certain fertilities by positioning stationary hooks appropriately. Given $\pi=\pi_1\cdots\pi_n\in S_n$, let $\widetilde\pi=(n+1)\pi(n+2)$. If $\pi_n=n$, put $\pi^*=\pi_1\cdots\pi_{n-1}\in S_{n-1}$. If $\lambda=\lambda_1\cdots\lambda_\ell\in S_\ell$ and $\mu=\mu_1\ldots\mu_m\in S_m$, then the \emph{sum} of $\lambda$ and $\mu$, denoted $\lambda\oplus\mu$, is obtained by placing the plot of $\mu$ above and to the right of the plot of $\lambda$. More formally, the $i^\text{th}$ entry of $\lambda\oplus\mu$ is \[(\lambda\oplus\mu)_i=\begin{cases} \lambda_i & \mbox{if } 1\leq i\leq \ell; \\ \mu_{i-\ell}+\ell, & \mbox{if } \ell+1\leq i\leq \ell+m. \end{cases}\] 

\begin{corollary}\label{Cor1}
Let $\ell$ and $m$ be positive integers. Let $\lambda=\lambda_1\cdots\lambda_\ell\in S_\ell$ and $\mu=\mu_1\ldots\mu_m\in S_m$, and assume $\lambda_\ell=\ell$. Letting $\pi=\lambda^*\oplus\widetilde\mu\in S_{\ell+m+1}$, we have \[|s^{-1}(\pi)|=|s^{-1}(\lambda)||s^{-1}(\mu)|.\]
\end{corollary}

\begin{proof}
Note that $\pi_\ell=\ell+m$ and $\pi_{\ell+m+1}=\ell+m+1$. The hook with southwest endpoint $(\ell,\ell+m)$ and northeast endpoint $(\ell+m+1,\ell+m+1)$ is a stationary hook of $\pi$. Following Proposition \ref{Prop1}, let $\sigma=\pi_1\cdots\pi_{\ell+1}\pi_{\ell+m+1}$ and $\tau=\pi_{\ell+1}\cdots\pi_{\ell+m}$. That proposition tells us that $|s^{-1}(\pi)|=|s^{-1}(\sigma)||s^{-1}(\tau)|$. We have $\tau_i=\mu_i+(\ell-1)$ for all $i\in\{1,\ldots,m\}$, so $\tau$ and $\mu$ are order isomorphic. It is immediate from the definition of the stack-sorting map that two permutations that are order isomorphic have the same fertility. Thus, $|s^{-1}(\tau)|=|s^{-1}(\mu)|$. Also, $\sigma$ is order isomorphic to the permutation $\lambda'=\lambda_1\cdots\lambda_{\ell-1}(\ell+1)\ell(\ell+2)$. We have \[\mathcal V(\lambda')=\{(q_0,\ldots,q_r,1):(q_0,\ldots,q_r)\in\mathcal V(\lambda)\}.\] According to Theorem \ref{Thm5}, \[|s^{-1}(\sigma)|=|s^{-1}(\lambda')|=\sum_{(q_0,\ldots,q_r,1)\in\mathcal V(\lambda')}C_{(q_0,\ldots,q_r,1)}=\sum_{(q_0,\ldots,q_r)\in\mathcal V(\lambda)}C_{(q_0,\ldots,q_r)}=|s^{-1}(\lambda)|. \qedhere\] 
\end{proof}

The following theorem is now an immediate consequence of Corollary \ref{Cor1}. 

\begin{theorem}\label{Thm1}
The set of fertility numbers is closed under multiplication.
\end{theorem}

The next theorem also follows easily from the above corollary. 

\begin{theorem}\label{Thm2}
If $f$ is a fertility number, then there are arbitrarily long permutations with fertility $f$. 
\end{theorem}

\begin{proof}
If $f$ is a fertility number, then there is a permutation $\lambda$ such that $|s^{-1}(\lambda)|=f$. We may assume that $\lambda$ is normalized. That is, $\lambda\in S_\ell$ for some $\ell\geq 1$. Now let $\mu=1\in S_1$. The permutation $\pi$ constructed in Corollary \ref{Cor1} has length $\ell+2$ and has fertility $f$. Repeating this procedure yields arbitrarily long permutations with fertility $f$.
\end{proof} 

Given a set $S$ of nonnegative integers, the quantity \[\liminf_{N\to\infty}\frac{|S\cap\{0,1,\ldots,N-1\}|}{N}\] is called the \emph{lower asymptotic density} of $S$. We next construct explicit permutations with certain fertilities in order to prove the following theorem. 

\begin{theorem}\label{Thm3}
Every nonnegative integer that is not congruent to $3$ modulo $4$ is a fertility number. The lower asymptotic density of the set of fertility numbers is at least $1954/2565\approx 0.7618$. 
\end{theorem}

\begin{proof}
We begin by showing that the permutation \[\xi_m=m(m-1)\cdots 321(m+1)(m+2)\cdots (2m)\] has fertility $2m$. The descent tops of this permutation are precisely the points of the form $(i,m+1-i)$ for $i\in\{1,\ldots,m-1\}$. In a valid hook configuration of $\xi_m$, the southwest endpoints of the hooks are precisely these descent tops. The northeast endpoints of hooks form an $(m-1)$-element subset of \linebreak $\{(m+1,m+1),\ldots,(2m,2m)\}$. Of course, this subset is determined by choosing the number $j\in\{1,\ldots,m\}$ such that $(m+j,m+j)$ is not in the subset. Once this number is chosen, the hooks themselves are determined by the fact that hooks cannot intersect in a valid hook configuration. The valid composition induced from this valid hook configuration is $(1,\ldots,1,2,1,\ldots,1)$, where the $2$ is in the $(m+1-j)^\text{th}$ position. Since $C_{(1,1,\ldots,1,2,1,\ldots,1)}=2$, it follows from Theorem \ref{Thm5} that $|s^{-1}(\xi_m)|=2m$. Thus, every even positive integer is a fertility number. This computation is illustrated in Figure \ref{Fig7} in the case $m=4$. 

\begin{figure}[t]
\begin{center}
\includegraphics[width=1\linewidth]{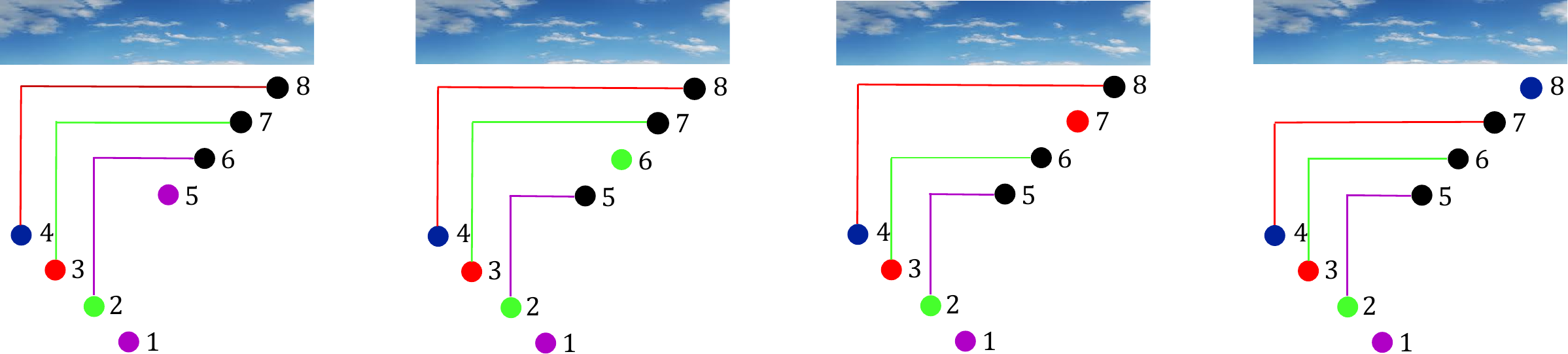}
\caption{The valid hook configurations of $\xi_4=43215678$ along with their induced colorings.}
\label{Fig7}
\end{center}  
\end{figure}

Suppose we have a permutation $\pi\in S_n$. Every valid hook configuration of $1\oplus\pi$ is obtained by placing a valid hook configuration of $\pi$ above and to the right of the point $(1,1)$. In the induced coloring of the plot of $1\oplus\pi$, the point $(1,1)$ must be blue. Every other point is given the same color as in the coloring of the plot of $\pi$ induced from the original valid hook configuration. It follows that \[\mathcal V(1\oplus\pi)=\{(q_0+1,q_1,\ldots,q_r):(q_0,\ldots,q_r)\in\mathcal V(\pi)\}.\] 

We have seen that the valid compositions of $\xi_m$ are precisely the compositions consisting of $m-1$ parts that are equal to $1$ and one part that is equal to $2$. Therefore, the valid compositions of $1\oplus\xi_m$ are \[(3,1,1,1,\ldots,1),\hspace{.2cm}(2,2,1,1,\ldots,1),\hspace{.2cm}(2,1,2,1,\ldots,1),\hspace{.2cm}\ldots,\hspace{.2cm}(2,1,1,\ldots,1,2).\] Invoking Theorem \ref{Thm5}, we find that \[|s^{-1}(1\oplus\xi_m)|=5+4(m-1)=4m+1.\] It follows that every positive integer that is congruent to $1$ modulo $4$ is a fertility number.

We saw in Example \ref{Exam1} that $27$ is a fertility number. The valid compositions of $1243567$ are $(5,1)$, $(4,2)$, and $(3,3)$, so \[|s^{-1}(1243567)|=C_{(5,1)}+C_{(4,2)}+C_{(3,3)}=42+28+25=95.\] This shows that $95$ is also a fertility number. If we combine Theorem \ref{Thm1} with the fact that all positive integers congruent to $1$ modulo $4$ are fertility numbers, then we find that all positive integers congruent to $3$ modulo $4$ that are multiples of $27$ or $95$ are also fertility numbers. In summary, every nonnegative integer $f$ satisfying one of the following conditions is a fertility number: 
\begin{itemize}
\item $f\not\equiv 3\pmod 4$;
\item $f\equiv 3\pmod 4$ and $27\mid f$;
\item $f\equiv 3\pmod 4$ and $95\mid f$.
\end{itemize}   
The natural density of the set of nonnegative integers satisfying one of these conditions is \[\frac 34+\frac{1}{4\cdot 27}+\frac{1}{4\cdot 95}-\frac{1}{4\cdot 27\cdot 95}=\frac{1954}{2565}. \qedhere\]
\end{proof}

The constant $1954/2565$ in Theorem \ref{Thm3} is not optimal. Indeed, we can increase the constant by simply exhibiting a fertility number that is congruent to $3$ modulo $4$ and is not already counted. Let us briefly describe one method for doing this. Let \[\zeta_m=(m+1)1(m+2)2(m+3)3\cdots(2m)m(2m+1)(2m+2)(2m+3).\] The valid compositions of $\zeta_m$ are precisely the compositions consisting of either one $3$ and $m$ $1$'s or two $2$'s and $m-1$ $1$'s. This is not difficult to see, but one can also give a rigorous proof using Theorem 2.4 from \cite{Defant5}. For example, $\zeta_2$ is the permutation $3142567$ from Example \ref{Exam1}. It follows from Theorem \ref{Thm5} that \[|s^{-1}(\zeta_m)|=5(m+1)+4{m+1\choose 2},\] and this is congruent to $3$ modulo $4$ whenever $m\equiv 2\pmod 4$. 
 
Proving that a given positive integer $f$ is a fertility number amounts to constructing a permutation with fertility $f$, as we did in the proof of Theorem \ref{Thm3}. Showing that a number is an infertility number is more subtle and requires additional tools. Bousquet-M\'elou introduced the notion of the \emph{canonical tree} of a permutation and showed that the shape of a permutation's canonical tree determines that permutation's fertility \cite{Bousquet}. She then asked for an explicit method for computing the fertility of a permutation from its canonical tree. The current author reformulated the notion of a canonical tree in the language of valid hook configurations, defining the \emph{canonical hook configuration} of a permutation \cite{Defant2}. He then described a theorem that yields an explicit method for computing a permutation's fertility from its canonical hook configuration. This result appears as Theorem 2.4 in the more recent article \cite{Defant5}. The following lemma is a consequence of this theorem; we omit the discussion describing how to compute the numbers $e_j$, $\mu_j$, and $\alpha_j$ because our present applications do not require it.   

\begin{lemma}\label{Lem2}
Let $\pi\in S_n$ be a permutation, and let $d_1<\cdots<d_k$ be the descents of $\pi$. There exist integers $e_0,\ldots,e_k,\mu_0,\ldots,\mu_k,\alpha_1,\ldots,\alpha_{k+1}$ (depending on $\pi$) with the following property. A composition $(q_0,\ldots,q_k)$ of $n-k$ into $k+1$ parts is a valid composition of $\pi$ if and only if the following two conditions hold:
\begin{enumerate}[(a)]
\item For every $m\in\{0,1,\ldots,k\}$, \[\sum_{j=m}^{e_m-1}q_j\geq\sum_{j=m}^{e_m-1}\mu_j.\]
\item If $m,p\in\{0,1,\ldots,k\}$ are such that $m\leq p\leq e_m-2$, then \[\sum_{j=m}^pq_j\geq d_{p+1}-d_m-\sum_{j=m+1}^{p+1}\alpha_j.\]
\end{enumerate}
\end{lemma}

Suppose $q=(q_0,\ldots,q_k)$, $q'=(q_0',\ldots,q_k')$, and $q''=(q_0'',\ldots,q_k'')$ are compositions of $n-k$ into $k+1$ parts (where $n$ and $k$ are as in Lemma \ref{Lem2}). We say $q$ \emph{interval dominates} $q'$ and $q''$ if \[\sum_{j=m_1}^{m_2}q_j\geq\min\left\{\sum_{j=m_1}^{m_2}q_j',\sum_{j=m_1}^{m_2}q_j''\right\}\quad\text{whenever }0\leq m_1\leq m_2\leq k.\] If $q',q''\in\mathcal V(\pi)$ and $q$ interval dominates $q'$ and $q''$, then it follows immediately from Lemma \ref{Lem2} that $q\in\mathcal V(\pi)$. In fact, this is the only reason why we need Lemma \ref{Lem2}. 

\begin{theorem}\label{Thm4}
The smallest fertility number that is congruent to $3$ modulo $4$ is $27$. 
\end{theorem}

\begin{proof}
We saw in Example \ref{Exam1} that $27$ is a fertility number. Assume by way of contradiction that there exists a fertility number $f\in\{3,7,11,15,19,23\}$. Let $n$ be the smallest positive integer such that there exists a permutation in $S_n$ with fertility $f$. Let $\pi\in S_n$ be one such permutation, and let $k$ be the number of descents of $\pi$. We say a composition $c$ \emph{has type} $\lambda$ if $\lambda$ is the partition formed by rearranging the parts of $c$ into nonincreasing order. For example, $(1,2,1,2)$ has type $(2,2,1,1)$. 

Because $|s^{-1}(\pi)|=f$ is odd, Theorem \ref{Thm5} tells us that $\pi$ must have a valid composition $q$ such that $C_q$ is odd. If any of the parts in $q$ were greater than $4$, the sum representing $|s^{-1}(\pi)|$ in Theorem \ref{Thm5} would be at least $42$, which is larger than $f$. If any of the parts were $2$ or $4$, $C_q$ would be even. This shows that all of the parts of $q$ are equal to $1$ or $3$. Furthermore, there is at most one part equal to $3$ (otherwise, the sum in Theorem \ref{Thm5} would be at least $25$). 

We know from Section 2 that every valid composition of $\pi$ is a composition of $n-k$ into $k+1$ parts. If $q=(1,1,\ldots,1)$, then $n=2k+1$. In this case, $(1,1,\ldots,1)$ is the only valid composition of $\pi$ (it is the only composition of $n-k$ into $k+1$ parts), so it follows from Theorem \ref{Thm5} that $|s^{-1}(\pi)|=1$. This is a contradiction, so $q$ must have type $(3,1,\ldots,1)$. Since $q$ is a composition of $n-k$, we must have $n=2k+3$. This implies that every composition of $n-k$ into $k+1$ parts is of type $(3,1,\ldots,1)$ or of type $(2,2,1,\ldots,1)$. Thus, every valid composition of $\pi$ is of one of these types.  

Let $Q_1,\ldots,Q_a$ be the valid compositions of $\pi$ of type $(3,1,\ldots,1)$, and let $b$ be the number of valid compositions of $\pi$ of type $(2,2,1,\ldots,1)$. By Theorem \ref{Thm5}, $5a+4b=f$. Reading this equation modulo $4$ shows that $a\equiv 3\pmod 4$. Since $f\leq 23$, we must have $a=3$. For $1\leq u<v\leq 3$, let $Q_{u,v}$ be the composition whose $i^\text{th}$ part is the arithmetic mean of the $i^\text{th}$ part of $Q_u$ and the $i^\text{th}$ part of $Q_v$. It is straightforward to see that $Q_{u,v}$ is a composition of $n-k$ into $k+1$ parts that has type $(2,2,1,\ldots,1)$ and that interval dominates $Q_u$ and $Q_v$. According to the discussion preceding this theorem, $Q_{1,2}$, $Q_{1,3}$, and $Q_{2,3}$ are valid compositions of $\pi$. Consequently, $b\geq 3$. It follows that $f=5a+4b\geq 27$, which is our desired contradiction. 
\end{proof}

Among the bulleted statements in the introduction, only the last one remains to be proven. The proof requires us to use Proposition \ref{Prop2}, which is stated below. The proof of this proposition relies on the following lemma, which is interesting in its own right.   

\begin{lemma}\label{Lem1}
Let $\pi$ be a sorted permutation with descents $d_1<\cdots<d_k$. Suppose there is an index $i\in\{1,\ldots,k\}$ such that $q_i=1$ for all $(q_0,\ldots,q_k)\in\mathcal V(\pi)$. If $H$ is a hook in a valid hook configuration of $\pi$ with southwest endpoint $(d_i,\pi_{d_i})$, then $H$ is a stationary hook of $\pi$.  
\end{lemma}

\begin{proof}
Recall from the previous section that we write valid hook configurations as tuples of hooks. Let $\mathcal H=(H_1,\ldots,H_k)$ be a valid hook configuration containing the hook $H$. Necessarily, we have $H=H_i$ (this is simply due to the conventions we chose in Section 2 concerning how to order hooks). Suppose by way of contradiction that there is a valid hook configuration $\mathcal H'=(H_1',\ldots,H_k')$ with $H_i'\neq H$. The southwest endpoint of $H_i'$ must be $(d_i,\pi_{d_i})$. Let $(j,\pi_j)$ and $(j',\pi_{j'})$ be the northeast endpoints of $H_i$ and $H_i'$, respectively. Without loss of generality, we may assume $j<j'$.

There exists $r\in\{i,\ldots,k\}$ such that $(d_{i+1},\pi_{d_{i+1}}),\ldots,(d_r,\pi_{d_r})$ are the descent tops of $\pi$ lying below $H$. Let \[\mathcal H''=(H_1',\ldots,H_i',H_{i+1},\ldots,H_r,H_{r+1}',\ldots,H_k').\] One can check that $\mathcal H''$ is a valid hook configuration of $\pi$. In the coloring of the plot of $\pi$ induced by $\mathcal H''$, both $(d_i+1,\pi_{d_i+1})$ and $(j,\pi_j)$ are given the same color as the hook $H_i'$. Letting $(q_0'',\ldots,q_k'')$ denote the valid composition of $\pi$ induced by $\mathcal H''$, we have $q_i''\geq 2$. This contradicts our hypothesis. 
\end{proof}

\begin{proposition}\label{Prop2}
Assume $n\geq 3$. Let $\pi=\pi_1\cdots\pi_n$ be a sorted permutation with descents $d_1<\cdots<d_k$. Suppose there is an index $i\in\{1,\ldots,k\}$ such that $q_i=1$ for all $(q_0,\ldots,q_k)\in\mathcal V(\pi)$. Let $\mathcal X=\{(q_0,\ldots,q_{i-1},q_{i+1},\ldots,q_k):(q_0,\ldots,q_k)\in\mathcal V(\pi)\}$. There exists a permutation $\zeta\in S_{n-2}$ such that $\mathcal V(\zeta)=\mathcal X$.   
\end{proposition}

\begin{proof}
According to Lemma \ref{Lem1}, $\pi$ has a stationary hook $H$ with southwest endpoint $(d_i,\pi_{d_i})$. Let $\lambda_{\text I},\lambda_{\text{II}},\lambda_{\text{III}},\lambda_{\text{IV}},\mu$ be the parts of the plot of $\pi$ as indicated in Figure \ref{Fig8}. Let us slide all of the points of $\lambda_{\text I}\cup\lambda_{\text{II}}\cup\mu$ up by some integral distance so that the lowest point of $\lambda_{\text I}\cup\lambda_{\text{II}}\cup\mu$ is now higher than the highest point of $\lambda_{\text{III}}\cup\lambda_{\text{IV}}$. We can then slide the points in $\lambda_{\text I}\cup\lambda_{\text{II}}$ up by another integral distance so that the lowest point in $\lambda_{\text I}\cup\lambda_{\text{II}}$ is now higher than the highest point in $\mu$. These two operations, illustrated in Figure \ref{Fig8}, produce a new permutation $\pi'$. 

Given a valid hook configuration of $\pi$, we obtain a valid hook configuration of $\pi'$ by keeping the hooks attached to their endpoints throughout these two sliding operations. Every valid hook configuration of $\pi'$ is obtained in this way because we can easily undo these sliding operations. Each valid hook configuration of $\pi$ induces a valid composition of $\pi$, and the corresponding valid hook configuration of $\pi'$ induces a valid composition of $\pi'$. These two valid compositions are identical because no points or hooks were ever moved horizontally and no hooks could have moved through each other during the sliding. Therefore, $\mathcal V(\pi)=\mathcal V(\pi')$. To ease notation, let us replace $\pi$ with this new permutation $\pi'$. In other words, we have shown that, without loss of generality, we may assume the plot of $\pi$ has the shape depicted in the rightmost part of Figure \ref{Fig8}.

\begin{figure}[t]
\begin{center}
\includegraphics[width=.75\linewidth]{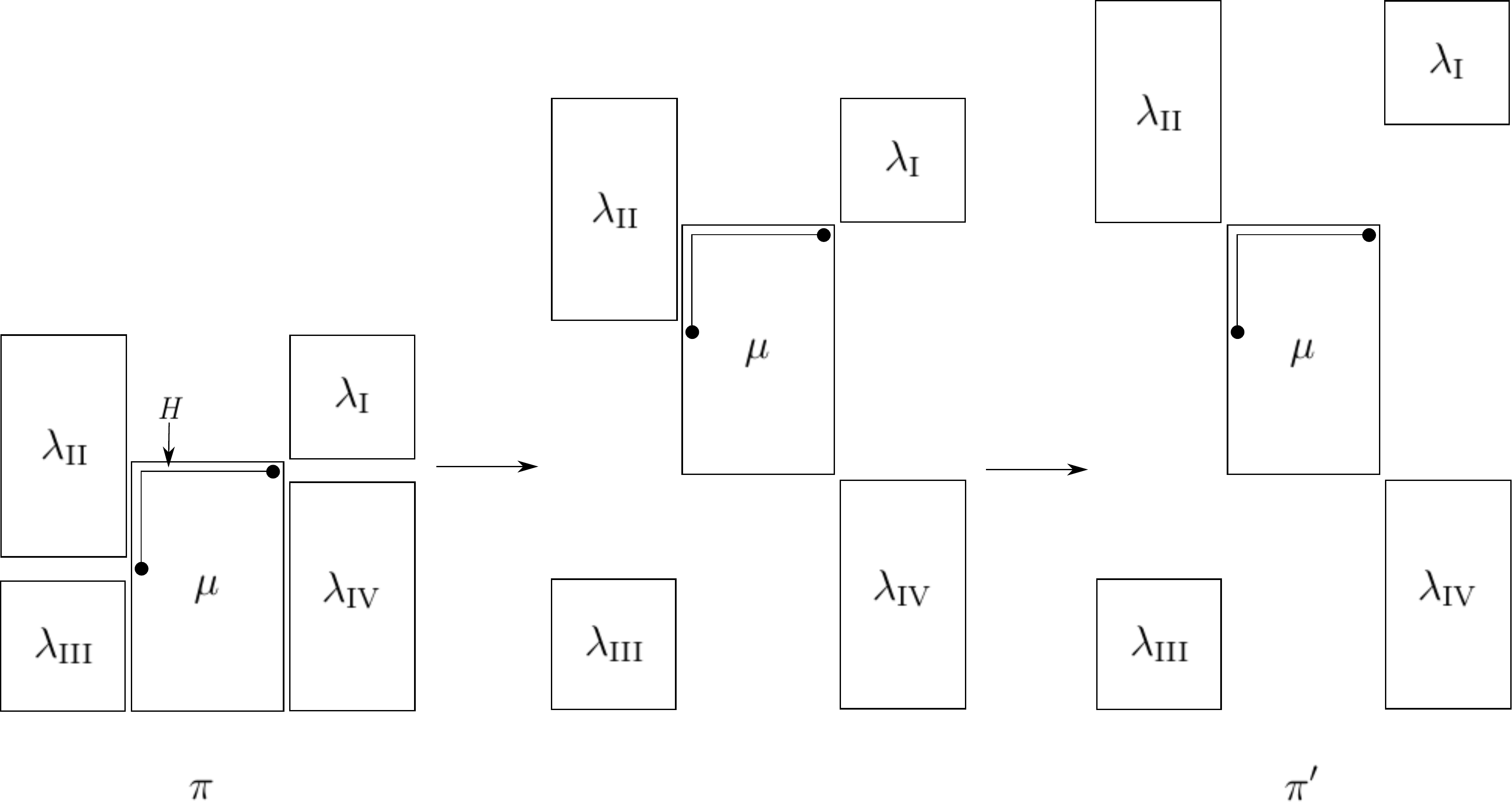}
\caption{The two sliding operations described in the proof of Proposition \ref{Prop2}.}
\label{Fig8}
\end{center}  
\end{figure}

Let us now remove the hook $H$ and its endpoints from the plot of $\pi$. After shifting the remaining points in $\mu$ to the left by $1$ and shifting the points in $\lambda_{\text{I}}\cup\lambda_{\text{IV}}$ left by $2$, we obtain the plot of a permutation $\xi$. We claim that $\mathcal V(\xi)=\mathcal X$. Indeed, there is a natural bijection $\varphi:\VHC(\pi)\to\VHC(\xi)$. To apply $\varphi$ to a valid hook configuration of $\pi$, we first leave unchanged every hook whose endpoints were not deleted (i.e., those hooks whose endpoints were not also endpoints of $H$). If there was a hook whose southwest endpoint was the northeast endpoint of $H$, replace its southwest endpoint with the rightmost remaining point from $\mu$. This is allowed because the rightmost remaining point in $\mu$ is a descent top of $\pi$ ($\lambda_{\text{IV}}$ lies below $\mu$). If there was a hook whose northeast endpoint was the southwest endpoint of $H$, replace its northeast endpoint with the leftmost remaining point from $\mu$. See Figure \ref{Fig9} for two examples of applications of $\varphi$. 

If $\mathcal H\in\VHC(\pi)$ induces a valid composition $(q_0,\ldots,q_k)\in\mathcal V(\pi)$, then $\varphi(\mathcal H)$ induces the valid composition $(q_0,\ldots,q_{i-1},q_{i+1},\ldots,q_k)\in\mathcal V(\xi)$. It follows that $\mathcal V(\xi)=\mathcal X$, as desired. Finally, we can normalize the permutation $\xi$ to obtain a permutation $\zeta\in S_{n-2}$ with $\mathcal V(\zeta)=\mathcal X$.  
\end{proof}

\begin{figure}[t]
\begin{center}
\includegraphics[width=.75\linewidth]{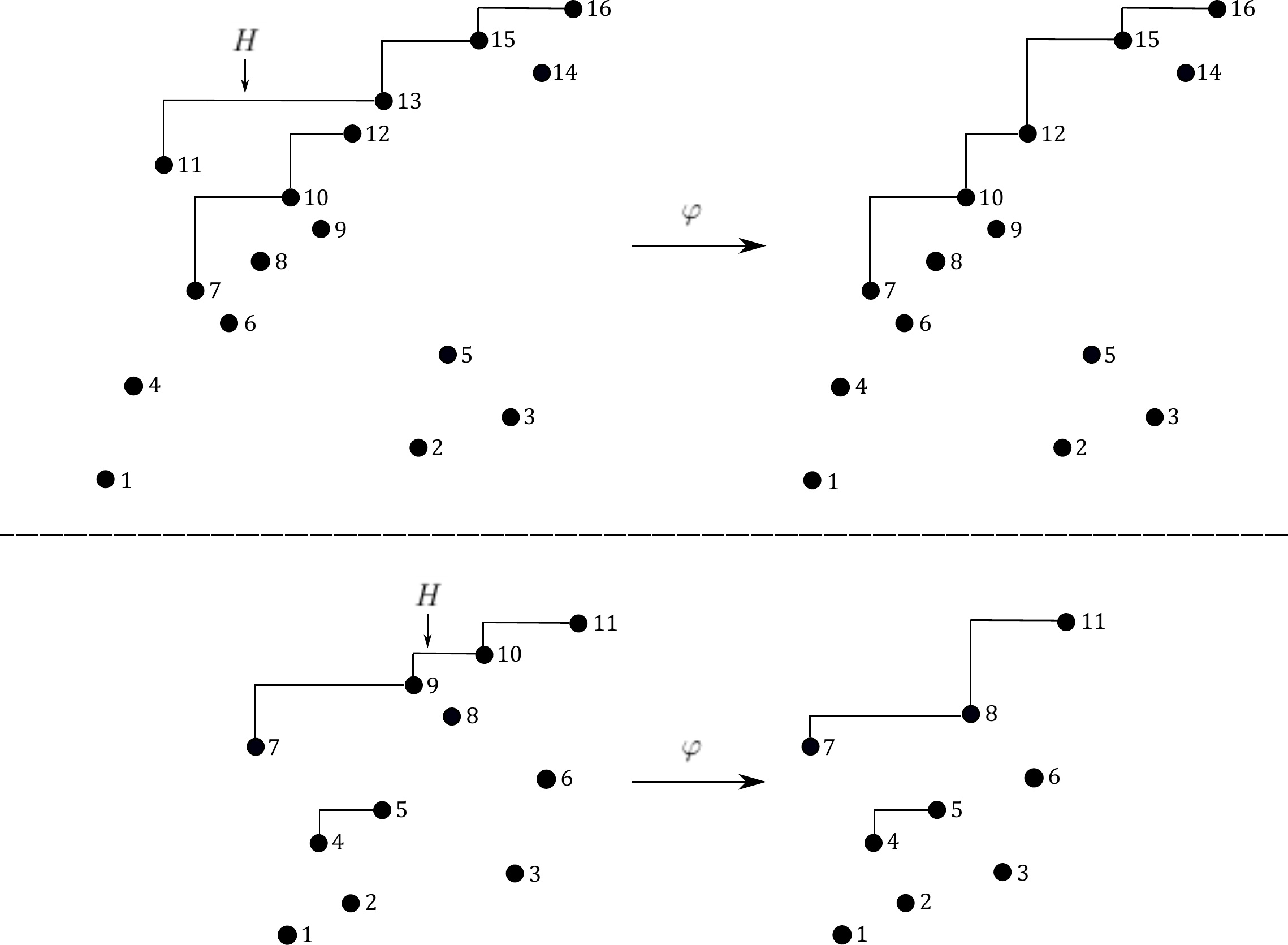}
\caption{Two example applications of the map $\varphi$ from the proof of Proposition \ref{Prop2}}
\label{Fig9}
\end{center}  
\end{figure}

The following corollary is now an immediate consequence of Theorem \ref{Thm5}.

\begin{corollary}\label{Cor2}
In the notation of Proposition \ref{Prop2}, the permutation $\zeta\in S_{n-2}$ has the same fertility as $\pi$. 
\end{corollary}

We can finally prove the last of our main theorems. As mentioned in the introduction, this theorem reduces the problem of determining whether a given positive integer is a fertility number to a finite problem. 

\begin{theorem}\label{Thm7}
If $f$ is a positive fertility number, then there exist a positive integer $n\leq f+1$ and a permutation $\pi\in S_n$ such that $f=|s^{-1}(\pi)|$. 
\end{theorem}

\begin{proof}
We know that there exist a positive integer $n$ and a permutation $\pi\in S_n$ such that $f=|s^{-1}(\pi)|$. Let us choose $n$ minimally. We will show that $n\leq f+1$. The theorem is easy when $f\in\{1,2\}$, so we may assume $f\geq 3$. This forces $n\geq 3$. 

Let $(q_{10},\ldots,q_{1k}),\ldots,(q_{m0},\ldots,q_{mk})$ be the valid compositions of $\pi$. \linebreak Form the $m\times (k+1)$ matrix  $M=(q_{i(j-1)})$ so that the rows of $M$ are precisely the valid compositions of $\pi$. If there is a column of $M$ whose entries are all $1$'s, then we can use Corollary \ref{Cor2} to see that there is a permutation in $S_{n-2}$ with fertility $f$, contradicting the minimality of $n$. Hence, every column of $M$ contains at least one number that is not $1$. 

Given an $a\times b$ matrix $D=(d_{ij})$ with positive integer entries, define \[N_D=b-1+\frac{1}{a}\sum_{i=1}^a\sum_{j=1}^b d_{ij}\] and \[F_D=\sum_{i=1}^aC_{(d_{i1},\ldots,d_{ib})}.\] From the fact that every valid composition of $\pi$ is a composition of $n-k$ into $k+1$ parts, we find that $N_M=n$. We know from Theorem \ref{Thm5} that $F_M=f$. Consequently, it suffices to prove the following claim. 

\noindent {\bf Claim:} If $D$ is a matrix with positive integer entries and every column of $D$ contains at least one number that is not $1$, then $N_D\leq F_D+1$.

To prove this claim, we first describe a useful reduction. We can choose an entry $d_{ij}\geq 2$ of $D$ and replace it with $d_{ij}-1$ to produce a new matrix $D'$. Note that $F_{D'}\leq F_D-1$ and $N_{D'}=N_D-1/a\geq N_D-1$. We can repeat this operation repeatedly until we are left with a matrix $D^*$ such that every entry of $D^*$ is either a $1$ or a $2$ and such that every column of $D^*$ contains exactly one $2$. If we performed the above operation $\ell$ times to obtain $D^*$ from $D$, then $F_{D^*}\leq F_D-\ell$ and $N_{D^*}=N_D-\ell/a\geq N_D-\ell$. It suffices to show that $N_{D^*}\leq F_{D^*}+1$. 

Let $u_i$ be the number of $2$'s in the $i^\text{th}$ row of $D^*$. Note that $u_1+\cdots+u_a=b$ because every column of $D^*$ has exactly one $2$. We have \[N_{D^*}=b-1+\frac{1}{a}(ab+u_1+\cdots+u_a)=\left(2+\frac 1a\right)(u_1+\cdots+u_a)-1\] and \[F_{D^*}+1=2^{u_1}+\cdots+2^{u_a}+1.\] We will show that 
\begin{equation}\label{Eq1}
\left(2+\frac 1a\right)(u_1+\cdots+u_a)-1\leq 2^{u_1}+\cdots+2^{u_a}+1
\end{equation} 
for every choice of nonnegative integers $u_1,\ldots,u_a$. 

If one of the integers $u_i$ is at least $3$, we can replace it by $u_i-1$. This has the effect of decreasing the expression on the left-hand side of \eqref{Eq1} by $2+1/a$ and decreasing the expression on the right-hand side by at least $4$. Therefore, it suffices to prove the inequality in \eqref{Eq1} after decreasing $u_i$ by $1$. We can repeatedly decrease the integers that are at least $3$ until every integer in the list $u_1,\ldots,u_a$ is at most $2$. In other words, it suffices to prove the inequality in \eqref{Eq1} under the assumption that $u_i\in\{0,1,2\}$ for all $i\in\{1,\ldots,a\}$. In this case, let $X_j=|\{i\in\{1,\ldots,a\}:u_i=j\}|$. With this notation, \eqref{Eq1} becomes \[\left(2+\frac{1}{X_0+X_1+X_2}\right)(X_1+2X_2)-1\leq X_0+2X_1+4X_2+1.\] This simplifies to \[\frac{-X_0+X_2}{X_0+X_1+X_2}\leq X_0+1,\] which obviously holds. 
\end{proof}

\section{Future Directions}

The primary objective of this article has been to gain an understanding of fertility numbers. Of course, the ultimate goal here is to obtain a complete description of all fertility numbers. This appears to be difficult, but there are less formidable problems whose solutions would still interest us greatly. For example, Theorem \ref{Thm3} leads us to ask the following question. 

\begin{question}\label{Quest1}
Does the set of fertility numbers have a natural density? If so, what is this natural density?  
\end{question}

We also have some conjectures spawning from our main theorems. 

\begin{conjecture}\label{Conj1}
There are infinitely many infertility numbers.
\end{conjecture}

The proof of Theorem \ref{Thm3} made use of the fact that $27$ and $95$ are fertility numbers. We saw in Theorem \ref{Thm4} that $27$ is the smallest fertility number that is congruent to $3$ modulo $4$, so we are led to make the following conjecture. 

\begin{conjecture}\label{Conj3}
The smallest fertility number that is congruent to $3$ modulo $4$ and is greater than $27$ is $95$. 
\end{conjecture}

It is desirable to have more efficient methods for determining whether or not a given positive integer is a fertility number. It is possible that such a method could arise by extending the techniques used in the proof of Theorem \ref{Thm4}. Such methods could certainly be useful for answering the above conjectures. This also leads to the problem of improving Theorem \ref{Thm7}. Given a fertility number $f$, let $\mathcal N(f)$ denote the smallest positive integer $n$ such that there exists a permutation in $S_n$ with fertility $f$. Theorem \ref{Thm7} states that $\mathcal N(f)\leq f+1$ for every fertility number $f$. We would like to have better estimates for $\mathcal N(f)$. In particular, we have the following conjecture. 

\begin{conjecture}\label{Conj4}
We have \[\lim_{f\to\infty}\mathcal N(f)/f=0,\] where the limit is taken along the sequence of positive fertility numbers.  
\end{conjecture}

Finally, recall that Theorem \ref{Thm1} tells us that the product of two fertility numbers is again a fertility number. We would like to have additional methods for combining fertility numbers in order to produce new ones. 

\section{Acknowledgments}
The author was supported by a Fannie and John Hertz Foundation Fellowship and an NSF Graduate Research Fellowship.

\end{document}